\documentclass[10pt]{amsart}

\usepackage{amsmath}
\usepackage{amsfonts}
\usepackage{amssymb}
\usepackage{graphicx}
\usepackage{color}
\usepackage{hyperref}
\usepackage{xcolor}

\newtheorem{theorem}{Theorem}[section]
\newtheorem{thm}[theorem]{Theorem}
\newtheorem*{theorem*}{Theorem}
\newtheorem{lemma}{Lemma}[section]
\newtheorem{corollary}[theorem]{Corollary}

\newtheorem{prop}{Proposition}[section]
\newtheorem{remark}[theorem]{Remark}

\def \b {\beta}

\def\Ric{\text{Ric}}

\def\a{\alpha}
\def\l{\lambda}

\def\g{\gamma}

\def\e{\epsilon}
\def\p{\partial}

\def\R{\mathbb{R}}

\def\vp{\varphi}

\def\k{\kappa}

\def\L{{\mathcal L}}

\def\w{\omega}

\def\Ric{\operatorname{Ric}}

\def\n{\nabla}

\numberwithin{equation}{section}

\begin{document}

\title[Lower bounds for the first eigenvalue on K\"ahler manifolds]{Lower bounds for the first eigenvalue of the Laplacian on K\"ahler manifolds}

%\title[Eigenvalue Estimates on K\"ahler Manifolds]{Eigenvalue Estimates on K\"ahler Manifolds }

\author{Xiaolong Li}
\address{(to start) Department of Mathematics and Statistics, 
McMaster University, Hamilton, Ontario, L8S 4K1, Canada}
%\address{Department of Mathematics, University of California, Irvine, CA 92697, USA}
\email{lxlthu@gmail}

\author{Kui Wang} \thanks{The research of the second author is supported by NSFC No.11601359} 
\address{School of Mathematical Sciences, Soochow University, Suzhou, 215006, China}
\email{kuiwang@suda.edu.cn}

%\author{Haotian Wu}%\thanks{The research of the third author is supported by ARC Grant DE180101348}
%\address{School of Mathematics and Statistics, The University of Sydney, NSW 2006, Australia}
%\email{haotian.wu@sydney.edu.au}

%\date{\today}
\subjclass[2010]{35P15, 53C55}

%35P30, 35K55, 49R05, 58C40, 58J50, 53C26, 53C55}
%\subjclass[2010]{32Q15, 35P30, 35K55, 49R05, 58C40, 58J50, 53C26, 53C55}
\keywords{Eigenvalue comparison, modulus of continuity, orthogonal Ricci curvature}

\begin{abstract}
We establish lower bound for the first nonzero eigenvalue of the Laplacian on a closed K\"ahler manifold in terms of dimension, diameter, and lower bounds of holomorphic sectional curvature and orthogonal Ricci curvature. 
On compact K\"ahler manifolds with boundary, we prove lower bounds for the first nonzero Neumann or Dirichlet eigenvalue in terms of geometric data. 
Our results are K\"ahler analogues of well-known results for Riemannian manifolds.  
%dimension, diameter, and lower bounds of holomorphic sectional curvature and orthogonal Ricci curvature.  We also lower bounds for the the first nonzero Neumann or Dirichlet eigenvalue of the Laplacian in terms of 
%We prove estimates on the modulus of continuity for solutions of a class of quasi-linear parabolic equations on a compact K\"ahler manifold in terms of initial oscillation, elapsed time, and lower bounds of holomorphic sectional curvature and orthogonal Ricci curvature.
%As an application, we derive lower bound for the first nonzero eigenvalue of the Laplacian on such a manifold in terms of the diameter. Analogous results are obtained for quaternion-K\"ahler manifolds with scalar curvature lower bounds as well.

\mbox{} 

\textit{Dedicated to Professor Richard Schoen on the occasion of his 70th birthday.}

\end{abstract}

\maketitle

\section{Introduction}
%Let's begin by recalling the eigenvalue problem for the Laplacian. 
Let $(M^n,g)$ be an $n$-dimensional compact connected Riemannian manifold, possibly with a smooth nonempty boundary $\p M$. 
Denote by $\Delta$ the Laplacian of the metric $g$.  In this paper, we shall consider the following three eigenvalue problems on the manifold $(M,g)$. When $\p M=\emptyset$, the closed eigenvalue problem seeks all real numbers $\l$ for which there are nontrivial solutions $u \in C^2(M)$ to the equation 
\begin{equation}\label{eq eigenvalue}
    \Delta u =-\l u.
\end{equation}
When $\p M \neq \emptyset$, the Dirichlet eigenvalue problem is to find all real numbers $\l$ for which there are nontrivial solutions $u \in C^2(M)\cap C^0(\overline{M})$ to \eqref{eq eigenvalue}, subject to the boundary condition
$$u=0 \text{ on } \p M,$$
and the Neumann eigenvalue problem looks for all real numbers $\l$ for which there are nontrivial solutions $u \in C^2(M)\cap C^1(\overline{M})$ to \eqref{eq eigenvalue}, subject to the boundary condition
$$\frac{\p u}{\p \nu} =0 \text{ on } \p M,$$
where $\nu$ denotes the outward unit normal vector fields of $\p M$. One can also consider these eigenvalue problems on a compact K\"ahler manifold by viewing it as a Riemannian manifold.

It is well-known that the spectrum of the Laplacian on a closed manifold consists of pure point spectrum $\{\mu_i\}_{i=0}^{\infty}$ that can be arranged in the order
\begin{equation*}
    0=\mu_0 < \mu_1 \leq \mu_2 \leq \mu_3 \leq \cdots  \to \infty.
\end{equation*}
The same holds for the spectrum with Neumann boundary condition. The Dirichlet spectrum consists of pure point spectrum $\{\l_i\}_{i=0}^{\infty}$ that can be arranged in the order
\begin{equation*}
    0 < \l_1 < \l_2 \leq \l_3 \leq \cdots  \to \infty.
\end{equation*}

%In this paper, we are interested in deriving lower bounds for the first nonzero eigenvalues, namely $\mu_1$ and $\l_1$, when the manifold is K\"ahler. 

The study of the first nonzero eigenvalues is an important issue in both mathematics and physics, since this constant determines the convergence rate of numerical schemes in numerical analysis, describes the energy of a particle in the ground state in quantum mechanics, and determines the decay rate of certain heat flows in thermodynamics. Given its physical and mathematical significance, numerous bounds have been established for the first nonzero eigenvalue of the Laplacian on a compact Riemannian manifold (with Dirichlet or Neumann boundary conditions if the boundary is nonempty). We refer the reader to the books \cite{Chavel84}\cite{SYbook} for classical results and the surveys \cite{LL10}\cite{Andrewssurvey15} for later developments. 

%Despite numerous lower bounds in terms of geometric data for the first nonzero eigenvalues for Riemannian manifolds (see for example \cite{Kasue84} and \cite{AC13} for the sharp ones), 

However, there are very few results for K\"ahler manifolds unless they are viewed as Riemannian manifolds. Lichnerowicz \cite{Lichnerowicz58} showed that if $M$ is closed K\"ahler manifold with $\Ric \geq (n-1)\k >0$, then $\mu_1 \geq 2(n-1)\k$. Note that this is a remarkable improvement of his well-known result in the Riemannian case, which asserts that $\mu_{1} \geq n\k$
if $M$ is a closed Riemannian manifold with $\Ric \geq (n-1)\k >0$. 
Both his results are sharp as they are achieved by the complex projective space with the Fubini-Study metric and the sphere with the round metric, respectively. The proofs are simple applications of the Bochner formula or its K\"ahler version \cite{Lichnerowicz58} (see also  \cite[Theorem 5.1]{Libook} for the Riemannian case and \cite[Theorem 6.14]{Ballmannbook} for the K\"ahler case, and also \cite{Urakawa87} for an alternative proof via harmonic maps in the K\"ahler case). Moreover, both Lichnerowicz's results hold for the first Dirichlet eigenvalue of the Laplacian if the boundary is convex. The Riemannian one is proved by Reilly \cite{Reilly79} and the K\"ahler one is a recent result of Guedj, Kolev  and Yeganefar \cite{GKY13}. Finally, we would like to mentioned that these results for K\"ahler manifolds have been generalized to the $p$-Laplacian by Blacker and Seto \cite{BS19} by establishing a $p$-Reilly formula. More precisely, they proved that if $M$ is a closed K\"ahler manifold (or with a convex boundary) with $\Ric \geq \k >0$, then the first nonzero eigenvalue (or first Dirichlet eigenvalue) of the $p$-Laplacian is bounded from below by $\left(\frac{p+2}{p(p-1)}(n-1)\k \right)^{\frac{p}{2}}$ provided that $p\geq2$. 

The above-mentioned results indicate that if one takes the K\"ahlerity into consideration, then the lower bound should be improved, as least in the positive curvature case.
In general, lower bounds that are sharp for Riemannian manifolds are not sharp for K\"ahler manifolds. Therefore, it is the purpose of this article to investigate the first nonzero eigenvalue of the Laplacian on a compact K\"ahler manifold and establish lower bounds that reflect the K\"ahler structure. 

%Obata \cite{Obata62} proved that if the equality occurs in Lichnerowicz's estimate in the Riemanning case, then $M$ is isometric to the round sphere. The equality in the K\"ahler case implies that the gradient of the eigenfunction is a holomorphic vector field, but it is not known whether $M$ is isometric to the complex projective space. 
%Instead of Ricci curvature lower bound, we will use lower bounds of holomorphic sectional curvature and orthogonal Ricci curvature (see Section 2 for the definitions), which are more suitable curvature conditions for K\"ahler manifolds as they reflect more on the K\"ahler structure. 

Lichnerowicz's Riemannian result is a special case of the following optimal lower bound on the first nonzero eigenvalue of the Laplacian in terms of dimension, diameter, and Ricci curvature lower bound.

\begin{thm}\label{Thm first eigenvalue for Laplacian}
Let $(M^n,g)$ be a compact Riemannian manifold (possibly with a smooth convex boundary) with diameter $D$ and $\Ric \geq (n-1)\k$ for $\k \in \R$. 
Let $\mu_1$ be the first nonzero eigenvalue of the Laplacian on $M$ (with Neumann boundary condition if $\p M \neq \emptyset$).
Then
\begin{equation*}
    \mu_1 \geq \bar{\mu}_1(n,\k,D),
\end{equation*}
where $\bar{\mu}_1(n,k,D)$ is the first nonzero Neumann eigenvalue of the one-dimensional eigenvalue problem:
\begin{equation*}
  \vp''-(n-1)T_\k \vp' =-\l \vp
\end{equation*}
on the interval $[-D/2,D/2]$.
\end{thm}
Here and in the rest of this article, the function $T_\k$ is defined for $\k \in \R$ by
\begin{equation}\label{def T}
     T_\k(t)=\begin{cases}
   \sqrt{\k} \tan{(\sqrt{\k}t)}, & \k>0, \\
   0, & \k=0, \\
   -\sqrt{-\k}\tanh{(\sqrt{-\k}t)}, & \k<0.
    \end{cases}
\end{equation}

The above theorem is proved by Zhong and Yang for the $\k=0$ case and by Kr\"oger \cite{Kroger92} for general $\k\in \R$ (see also Bakry and Qian \cite{BQ00} for the above explicit form and extensions to smooth metric measure spaces). These works use the gradient estimates method initiated by Li \cite{Li79} and Li and Yau \cite{LY80}. In 2013, Andrews and Clutterbuck \cite{AC13} gave a simple proof via the modulus of continuity estimates (see also \cite{WZ17} for an elliptic proof based on \cite{AC13} and \cite{Ni13}). 

The first main theorem of this paper is the following K\"ahler analogue of the above theorem. 
%Our first main theorem states: 
\begin{thm}\label{thm main}
Let $(M^m, g, J)$ be a compact K\"ahler manifold of complex dimension $m$ and diameter $D$ whose holomorphic sectional curvature is bounded from below by $4\k_1$ and orthogonal Ricci curvature is bounded from below by $2(m-1)\k_2$ for some $\k_1,\k_2 \in \R$.    
Let $\mu_{1}$ be the first nonzero eigenvalue of the Laplacian on $M$ (with Neumann boundary condition if $M$ has a strictly convex boundary). Then 
\begin{equation*}
    \mu_{1} \geq \bar{\mu}_{1}(m,\k_1,\k_2,D),
\end{equation*}
where $\bar{\mu}_{1}(m,\k_1,\k_2,D)$ is the first Neumann eigenvalue of the one-dimensional eigenvalue problem %$\L$ defined in \eqref{eq 1D} 
\begin{equation*}
    \vp''-\left(2(m-1)T_{\k_2}+T_{4\k_1} \right)\vp' =-\l \vp
\end{equation*}
on $[-D/2,D/2]$.
%\begin{equation}\label{1D Eigenvalue}
%\begin{cases}
%  \vp'' -\left(2(m-1)T_{\k}+ T_{4 \k }\right) \vp'=-\l \vp, \\
%\end{cases}
%\end{equation}
\end{thm}

Instead of lower bounds of Ricci curvature, 
we use lower bounds of holomorphic sectional curvature and orthogonal Ricci curvature (see Section 2 for definitions) in Theorem \ref{thm main}. These curvature conditions are more suitable for K\"ahler manifolds as they reflect more on the K\"ahler structure, particularly for various comparison theorems.
For example, the Laplacian comparison theorem in Riemannian geometry assumes only Ricci lower bounds and are sharp on space forms of constant sectional curvature, while the Laplacian comparison theorem on K\"ahler manifolds that are sharp on complex space forms of constant holomorphic sectional curvature requires lower bounds of both holomorphic sectional curvature and orthogonal Ricci curvature. The latter was established until recently by Ni and Zheng in \cite{NZ18} and was known previously under the stronger condition of bisectional curvature lower bounds by the work of Li and Wang \cite{LW05} (see also \cite{TY12} for Hessian comparison theorem under bisectional curvature lower bounds). 
Since both Theorem \ref{Thm first eigenvalue for Laplacian} and \ref{thm main} can be viewed as eigenvalue comparison theorems, it is natural to work with lower bounds of holomorphic sectional curvature and orthogonal Ricci curvature.

Our proof of Theorem \ref{thm main} uses the modulus of continuity approach introduced by Andrews and Clutterbuck in \cite{AC13}. Lower bounds of the first nonzero eigenvalue are derived as large time implication of the modulus of continuity estimates, which in turn relies on a comparison theorem for the second derivatives of $d(x,y)$ as a function on $M\times M$.  
%proof of Theorem \ref{Thm first eigenvalue for Laplacian} 
%Since our approach to prove Theorem \ref{thm main} is also via the modulus of continuity estimates, it is thus more natural to work with lower bounds of holomorphic sectional curvature and orthogonal Ricci curvature.

To the best of our knowledge, Theorem \ref{thm main} is the first lower bound for the first nonzero eigenvalue of the Laplacian on K\"ahler manifolds that depends on the diameter. It also seems to be the first lower bound for the Neumann boundary condition in the K\"ahler setting. 
%, unless $M$ is viewed as a Riemannian manifold and one applies the results in the Riemannian setting. 
Moreover, the monotonicty of $\bar{\mu}_1(m,\k_1,\k_2,D)$ in $D$ implies that when the diameter is small, the lower bound provided in Theorem \ref{thm main} is better than any lower bound that is independent of the diameter. 
%For many types of comparison theorems for K\"ahler manifolds, the holomorphic sectional curvature, orthogonal Ricci curvature, and bisectional curvature are more suitable curvature conditions as they reflect more on the K\"ahler structure and are needed to obtain sharp results, see \cite{LW05}\cite{NZ18}\cite{TY12}. 

%\red{provide some explicit bounds}

Expressing $\bar{\mu}_1(m,\k_1,\k_2,D)$ in terms of elementary functions does not seem to be possible, so we provide some explicit bounds and estimates in the next proposition for the purpose of comparing with other results. 
\begin{prop}\label{prop explicit bounds}
Let $\bar{\mu}_1(m,\k_1,\k_2,D)$ be as in Theorem \ref{thm main}. Then we have 
\begin{equation*}
    \bar{\mu}_1(m,\k_1,\k_2,D) 
    \begin{cases} 
    = \frac{\pi^2}{D^2}, & \text{ if } \k_1=\k_2=0;\\
    \geq  \bar{\mu}_1\left(m,\k_1,0, \frac{\pi}{2\sqrt{\k_1}}\right)=8\k_1, & \text{ if } \k_1 >0, \k_2=0; \\
     \geq  \bar{\mu}_1\left(m,0,\k_2, \frac{\pi}{\sqrt{\k_1}}\right)=(2m-1)\k_2,  & \text{ if } \k_1=0, \k_2>0. \\
     %>  \bar{\sigma}_1(2m,\k, D), & \text{ if } \k_1=\k_2=\k <0.
    \end{cases}
\end{equation*}
%Here $\bar{\sigma}_1(2m,\k, D)$ is the first nonzero Neumann eigenvalue of the operator $$L \vp=\vp''-(2m+1)T_{\k}\vp'$$ on $[-D/2,D/2]$.
\end{prop}

Theorem \ref{thm main} is sharp when $\k_1=\k_2=0$. In this case, we have $\mu_1 \geq \bar{\mu}_1=\frac{\pi^2}{D^2}$, giving the same lower bound as in Zhong and Yang \cite{ZY84} for Riemannian manifolds with nonnegative Ricci curvature. The sharpness can be seen by constructing a sequence of K\"ahler manifolds with nonnegative holomorphic sectional curvature and nonnegative orthogonal sectional curvature, which geometrically collapse to the interval $[-D/2,D/2]$. For example, one can take the examples constructed in \cite[Section 5]{AC13} with $\k=0$ for such K\"ahler manifolds (also notice that the examples in \cite{AC13} are not K\"ahler if $\k \neq 0$). 
We suspect that Theorem \ref{thm main} is sharp for $\k_1=\k_2 \neq 0$ as well.

Next let's turn to the first Dirichlet eigenvalue $\l_1$. For convenience, denote by $C_{\kappa, \Lambda}(t)$ the unique solution of the initial value problem
\begin{equation}\label{C def}
    \begin{cases} 
    \phi''+\kappa \phi =0, \\
    \phi(0)=1,     \phi'(0) =-\Lambda,
    \end{cases}
\end{equation}
and define $T_{\kappa,\Lambda}$ for $\kappa, \Lambda \in \R$ by 
\begin{equation}\label{def T kappa Lambda}
    T_{\kappa,\Lambda} (t):=- \frac{C'_{\kappa, \Lambda}(t)}{C_{\kappa, \Lambda}(t)}.
\end{equation}
In the Riemannian setting, the following result is well-known.
\begin{thm}\label{Thm B}
Let $(M^n,g)$ be a compact Riemannian manifold with smooth boundary $\p M \neq \emptyset$. 
Suppose that the Ricci curvature of $M$ is bounded from below by $(n-1)\kappa$ and the mean curvature of $\p M$ is bounded from below by $(n-1)\Lambda$ for some $\kappa, \Lambda \in \R$. 
Let $\l_1$ be the first Dirichlet eigenvalue of the Laplacian on $M$. Then
\begin{align*}
   & \l_{1} \geq \bar\l_1(n,\k,\Lambda,R) , 
\end{align*}
where $\bar\l_1(n,\k,\Lambda,R)$
is the first eigenvalue of the one-dimensional eigenvalue problem 
\begin{equation}\label{eq 1.4}
    \begin{cases}
   \vp'' -(n-1)T_{\kappa, \Lambda} \vp' =-\l\vp, \\
    \vp(0)=0,     \vp'(R)=0.
    \end{cases}
\end{equation}
Here $R$ denotes the inradius of $M$ defined by
\begin{equation*}
    R=\sup \{d(x,\p M) :x \in \p M \}.
\end{equation*}
%Moreover, the equality occurs if and only if $(M,g)$ is a $(\kappa, \Lambda)$-model space defined in Definition
\end{thm}
Regarding Theorem \ref{Thm B}, the special case $\k=\Lambda=0$ is due to Li and Yau \cite{LY80} and the general case is obtained by Kause \cite{Kasue84}.

Our second main result is a K\"ahler counterpart of the above theorem.
\begin{thm}\label{Thm C}
Let $(M^m,g,J)$ be a compact K\"ahler manifold with smooth nonempty boundary $\p M$. 
Suppose that the holomorphic sectional curvature is bounded from below by $4\k_1$ and orthogonal Ricci curvature is bounded from below by $2(m-1)\k_2$ for some $\k_1,\k_2 \in \R$, and the second fundamental form on $\p M$ is bounded from below by $\Lambda \in \R$. Let $\l_1$ be the first Dirichlet eigenvalue of the Laplacian on $M$. Then 
$$\l_1\ge\bar{\l}_1(m, \k_1,\k_2, \Lambda, R)$$
where $\bar{\l}_1(m, \k_1,\k_2, \Lambda, R)$ is the first eigenvalue of the one-dimensional eigenvalue problem
\begin{equation}\label{eq 1.5}
    \begin{cases}
   \vp''-\left(2(m-1)T_{\k_2, \Lambda}+T_{4\k_1, \Lambda} \right)\vp'  =-\l\vp, \\
    \vp(0)=0,     \vp'(R)=0.
    \end{cases}
\end{equation}
\end{thm}

When the boundary is convex, namely $\Lambda=0$, it is easily seen that we have 
\begin{equation*}
  \bar{\l}_1(m,\k_1,\k_2,0,R) =\bar{\mu}_1(m,\k_1,\k_2,R).
\end{equation*}
Thus we can obtain the same explicit lower bounds and estimates as in Proposition \ref{prop explicit bounds} for $\bar{\l}_1(m,\k_1,\k_2,0,R)$.

The proof of Theorem \ref{Thm C} is similar to that of Theorem \ref{Thm B} given in \cite{Kasue84} and it relies on a comparison theorem for the second derivatives of $d(x,\p M)$ and Barta's inequality. 
With the help of a generalized Barta's inequality for the $p$-Laplacian (see \cite[Theorem 3.1]{LW20a}), the same argument indeed yields such lower bounds for the first Dirichlet eigenvalue of the $p$-Laplacian for all $1<p<\infty$.

The paper is organized as follows. In Section 2, we recall the definitions of holomorphic sectional curvature and orthogonal Ricci curvature for K\"ahler manifolds. Section 3 is devoted to proving the modulus of continuity estimate for a large class of quasilinear parabolic equations, in terms of dimension, diameter and curvature lower bounds. We prove Theorem \ref{thm main} in Section 4, as large time implication of the modulus of continuity estimates. 
Explicit lower bounds and estimates in Proposition \ref{prop explicit bounds} will be proved in Section 5. 
In Section 6, we prove a comparison theorem for the second derivatives of $d(x,\p M)$, which will be used to prove Theorem \ref{Thm C} in Section 7.

%The results we obtain in this paper are K\"ahler analogue of many results in the Riemannian setting, but with improved estimates due to K\"ahlerity. 

%Many results have been successfully extended to the nonlinear $p$-Laplacian during the last two decades, (see \cite{Matei00}\cite{NV14}\cite{Sakurai19}\cite{LTW20} and the references therein). 

\section{Curvatures of K\"ahler Manifolds}
In this section, we briefly recall the notions holomorphic sectional curvature and orthogonal Ricci curvature for K\"ahler manifolds. 

Let $(M^m,g, J)$ be a K\"ahler manifold of complex dimension $m$ (real dimension is $n=2m$).
We regard $M$ as a $2n$-dimensional Riemannian manifold with a parallel complex structure $J$. A plane $\sigma \subset T_pM$ is said to be holomorphic if it is invariant by the complex structure tensor $J$. The restriction of the sectional curvature to holomorphic planes is called the holomorphic sectional curvature and will be denoted by $H$. In other words,
if $\sigma$ is a holomophic plane spanned by $X$ and $JX$, then the holomorphic sectional curvature of $\sigma$ is defined by 
%$H(\sigma)$ is defined only when $\sigma$ is invariant by $J$, and 
\begin{equation*}
    H(\sigma) := H(X)=\frac{R(X,JX,X,JX)}{|X|^4}.
\end{equation*}
%$H(\sigma) = R(X,JX,X,JX)$ . 
%If $X$ is a vector in $\sigma$, we shall also write $H(X)$ for $H(\sigma)$. 
We say the holomorphic sectional curvature is bounded from below by $\k \in \R$ (abbreviated as $H\geq \k$) if 
$H(\sigma) \geq \kappa$ for all holomorphic planes $\sigma \subset T_pM$ and all $p\in M$.  

%Another frequently used curvature condition is the so-called bisectional curvature. 
%Given two $J$-invariant planes $\sigma_1$ and $\sigma_2$ in $T_pM$, we define the (holomorphic) bisectional curvature $B(\sigma_1, \sigma_2)$ by
%\begin{equation*}
%    B(\sigma_1, \sigma_2) = R(X, JX, Y, JY).
%\end{equation*}
%where $X$ is a unit vector in $\sigma_1$ and $Y$ a unit vector in $\sigma_2$.
%Following \cite{LW05} and \cite{TY12}, we say the holomorphic bisectional curvature is bounded from below by $\k \in \R$ (abbreviated as $B\geq \k$) if 
%\begin{equation*}
%    R(X, JX, Y, JY) \geq \k \left(1+  \right)
%\end{equation*}

The orthogonal Ricci curvature, denoted by $\Ric^\perp$, is defined for any $X\in T_pM$ by
\begin{equation*}
    \Ric^{\perp}(X,X) =\Ric(X,X)-H(X)|X|^2.
\end{equation*}
We say the orthogonal Ricci curvature is bounded from below by $\k \in \R$ (abbreviated as $\Ric^\perp \geq \k$) if $\Ric^{\perp}(X,X) \geq \k |X|^2$ for all $X\in T_pM$ and all $p\in M$.
This new curvature $\Ric^{\perp}$ was recently introduced by Ni and Zheng \cite{NZ18} in the study of Laplace comparison theorems on K\"ahler manifolds.  We refer the readers to \cite{NZ18}\cite{NZ19}\cite{NWZ18} for a more detailed account on $\Ric^{\perp}$ and recent developments.

\section{Modulus of Continuity Estimates on K\"ahler Manifolds}

In the section, we derive modulus of continuity estimates for a large class of quasilinear parabolic equations on K\"ahler manifolds, in terms of initial oscillation, elapsed time, and lower bounds of holomorphic sectional curvature and orthogonal Ricci curvature. One should compare it with its Riemannian version derived by Andrews and Clutterbuck in \cite{AC13}. 

Recall that the modulus of continuity $\w$ of a continuous function $u$ defined on a metric space $(X,d)$ is defined by 
\begin{equation*}
   \w(s):= \sup \left\{ \frac{u(y)-u(x)}{2} : d(x,y)=2s \right\}.
\end{equation*} 
In a series of papers \cite{AC09a, AC09, AC11, AC13, Andrewssurvey15}, Andrews and Clutterbuck investigated the question of how the modulus of continuity $\w$ of $u$ evolves when $u$ is evolving by a parabolic partial differential equation on a bounded domain in the Euclidean space or on a compact Riemannian manifold. They managed to prove that 
for a large class of quasi-linear parabolic equations (see \eqref{eq isotropic} below), the modulus of continuity of a solution is a subsolution of the associated one-dimensional equations.  This in particular allows them to proved in \cite{AC09} gradients estimates for the graphical solutions of smoothly anisotropic mean curvature flows depending only the initial oscillation, whereas this has not yet been accomplished using direct estimates on the gradient. However, it has turned out that the most important applications of the modulus of continuity estimates are the large time implications. For instance, the modulus of continuity estimate leads directly to exponential decay rate for the oscillation of solutions of the heat equation, thus implying a lower bound on the first nonzero eigenvalue of the Laplacian. 
This observation was used in \cite{AC13} (see also \cite{WZ17} for an elliptic proof) to provide a simple alternative proof of the sharp lower bound for the first eigenvalue of the Laplacian in terms of dimension, diameter and lower bound of Ricci curvature (see Theorem \ref{Thm first eigenvalue for Laplacian}). 
More surprisingly, this idea was used in \cite{AC11} to derive an optimal lower bound on the difference between the first two Dirichlet eigenvalues for the Laplacian on a convex Euclidean domain, thus
proving the long-standing Fundamental Gap Conjecture (see also \cite{Ni13} for an elliptic proof and \cite{HWZ20}\cite{DSW18}\cite{SWW19} for fundamental gap of convex domains in the sphere). In another direction, the modulus of continuity estimate has been extended to viscosity solutions in \cite{Li16}\cite{LW17}, and to fully nonlinear parabolic equations in \cite{Li20}. 
We refer the read to the the survey by Andrews \cite{Andrewssurvey15}, where these ideas were further explained and connections to other problems in geometric analysis are made.

As in \cite{AC13} and \cite{Andrewssurvey15}, we consider the following quasilinear isotropic equation
\begin{equation}\label{eq isotropic}
\frac{\p u}{\p t} =Q[u]:=\left[\a(|
\n u|) \frac{\n_iu \n_j u}{|\n u|^2} +\b(|\n u|)\left(\delta_{ij}-\frac{\n_i u \n_j u}{|\n u|^2}\right) \right].   
\end{equation}
Here $\a$ and $\b$ are smooth positive functions. Some important examples of \eqref{eq isotropic} are the heat equation (with $\a=\b=1$) and the $p$-Laplacian heat flows (with $\a=(p-1)|\n u|^{p-2}$ and $\b=|\n u|^{p-2}$) and the graphical mean curvature flow (with $\a=1/(1+|\n u|^2)$ and $\b=1$). 
%, and the game-theoretic $p$-Laplacian (with $\a=(p-1)$ and $\b=1$).  

To the operator $Q$ defined in \eqref{eq isotropic}, the associated one-dimensional operator $\L$ is given by \begin{equation}\label{eq 1D}
    \L \vp =\a(\vp')\vp'' -\left( 2(m-1)T_{\k_2}+T_{4\k_1}\right)\b(\vp')\vp', 
\end{equation}
where $T_{\k}$ is defined by  \eqref{def T}.
%\begin{equation}\label{def T}
%    T_\k(t) :=%\k \frac{s_\k(t)}{c_\k(t)} =
%    \begin{cases} 
%    \sqrt{\k}\tan \sqrt{\k} t & \text{ if } \k >0, \\
%    0 & \text{ if } \k =0,\\
%    -\sqrt{-\k} \tanh \sqrt{-\k }t & \text{ if } \k <0.
%    \end{cases}
%\end{equation}

%More precisely, 

The main result of this section is the following modulus of continuity estimates on K\"ahler manifolds. 

%establish modulus of continuity estimates for solutions of parabolic equations on a compact K\"ahler manifold in terms of initial oscillation, elasped time, and lower bounds of holomorphic sectional curvature and orthogonal Ricci curvature. More precisely, we prove that 

%Moreover, the estimates are improved due to K\"ahlerity or quaternion-K\"ahlerity. 

%Our first main result says that the modulus of continuity $\w$ of a solution $u$ of \eqref{eq isotropic} on a K\"ahler manifold is a subsolution of the one-dimensional equation $\w_t =\L \w$. More precisely

%is the following modulus of continuity estimates.   

\begin{thm}\label{thm MC}
Let $(M^m, g, J)$ be a compact K\"ahler manifold with diameter $D$ whose holomorphic sectional section is bounded from below by $4\k_1$ and the orthogonal Ricci curvature is bounded from below by $2(m-1)\k_2$ for some $\k_1, \k_2 \in \R$.  
Let $u : M \times [0,T) \to \R$ be a solution of \eqref{eq isotropic} (with Neumann boundary condition if $M$ has a strictly convex boundary). Then the modulus of continuity $\omega:[0,D/2] \times [0,T) \to \R$ of $u$ is a viscosity subsolution of the one-dimensional equation
\begin{equation}\label{ODE K}
    \w_t =\L \w,
\end{equation}
where $\L$ is defined in \eqref{eq 1D}. 
\end{thm}

\begin{proof}[Proof of Theorem \ref{thm MC}]
The proof is in the same spirit as in \cite{AC13}, but differs from the choices of variations due to K\"ahlerity and different assumptions on the curvatures. 
It will be convenient to use the function $c_\k$ defined by
\begin{equation}\label{def c kappa}
    c_\k(t) =\begin{cases} 
    \cos \sqrt{\k} t & \text{ if } \k >0, \\
    1 & \text{ if } \k =0,\\
    \cosh \sqrt{-\k t} & \text{ if } \k <0.
    \end{cases}
\end{equation}
By the definition of viscosity solutions (see \cite{CIL92}), we need to show that for every smooth function $\vp$ that touches $\w$ from above at $(s_0,t_0) \in (0,D/2)\times (0,T)$ in the sense that 
\begin{equation*}
    \begin{cases}
    \vp (s,t)  \geq  \w(s,t)  \text{  near } (s_0,t_0),\\ 
    \vp(s_0,t_0) =  \w(s_0,t_0), 
    \end{cases}
\end{equation*}
it holds that  
\begin{equation}\label{eq2.0}
\vp_t \leq \L  \vp, 
\end{equation}
at the point $(s_0,t_0)$. 
It follows from the definition of $\w$ that for such a function $\vp$, we have
\begin{equation} \label{eq2.1}
    u\left(\gamma(1),t \right)- u\left( \gamma(0),t \right) -2 \vp \left( \frac{L[\gamma]}{2}, t \right) \leq 0
\end{equation}
for any $t \leq t_0$ close to $t_0$ and any smooth path $\gamma: [a,b] \to M$ with length close to $2s_0$. Moreover, since $M$ is compact, there exist points $x_0$ and $y_0$ in $M$ (assume for a moment that $\p M = \emptyset$), with $d(x_0,y_0)=2s_0$ such that the equality in \eqref{eq2.1} holds for $\gamma_0:[-s_0,s_0] \to M$, a length-minimizing unit speed geodesic connecting $x_0$ and $y_0$.
The key idea is to derive useful inequalities from the first and second tests along smooth family of variations of the curve $\gamma_0$. 
For this purpose, we need to recall the first and second variation formulas of arc length.  If $\gamma:(r,s) \to \gamma_r(s)$ is a smooth variation of $\gamma_0(s)$, then we have 
\begin{equation*}\label{eq 1st variation}
   \left. \p_rL[\gamma_r]\right|_{r=0} = \left.g(T, \gamma_r) \right|_{-s_0}^{s_0}, 
\end{equation*}
and 
\begin{equation*}\label{eq 2nd variation}
    \left. \p_r^2 L[\gamma_r]\right|_{r=0} = \int_{-s_0}^{s_0} \left(|(\n_s \gamma_r)^\perp|^2 -R(\gamma_s,\gamma_r,\gamma_s,\gamma_r)\right)ds +\left. g(T,\n_r \gamma_r)\right|_{-s_0}^{s_0},
\end{equation*}
where $T$ is the unit tangent vector to $\gamma_0$. 

It will be convenient to work in the (complex) Fermi coordinates along $\gamma_0$ chosen as follows. We pick an orthonormal basis $\{e_i\}_{i=1}^{2m}$ for $T_{x_0}M$ with $e_1=\gamma_0'(-s_0)$ and $e_2 =J \gamma_0'(-s_0)$, where $J$ is the complex structure. 
Then parallel transport along $\gamma_0$ produces an orthonormal basis $\{e_i(s)\}_{i=1}^{2m}$ for $T_{\gamma_0(s)}M$ with $e_1(s)=\gamma_0'(s)$ for each $s\in [-s_0,s_0]$. Since $J$ is parallel and $\gamma_0$ is a geodesic, the vector field $J\gamma_0'(s)$ is parallel and thus we have $e_2(s) =J e_1(s)$ for each $s\in [-s_0, s_0]$. 

First derivatives consideration implies that 
%The time derivative inequality gives 
\begin{equation}\label{eq2.2}
    Q[u](y_0,t_0) -Q[u](x_0,t_0) -2\vp_t \geq 0,
\end{equation}
and
\begin{align}\label{eq2.3}
    \n u(x_0,t_0)  =-\vp'e_1(-s_0), \text{\quad}
    \n u(y_0,t_0) =\vp'e_1(s_0).
\end{align}
Here and below, all derivatives of $\vp$ are evaluated at $(s_0,t_0)$. 
The variation $\gamma(r,s)=\gamma_0\left(s+r\frac{2s-1}{L} \right)$ has $\gamma_r= \frac{2s-1}{L}\gamma_0'$, and thus $\p_r L[\gamma]=2$, and $\p_r^2 L[\gamma] =0$. The second derivative test for this variation produces 
\begin{equation}\label{eq2.6}
    u_{11}(y_0,t_0)-u_{11}(x_0,t_0) - 2\vp'' \leq 0.
\end{equation}
Here the subscripts denote covariant derivatives in directions corresponding to the orthonormal basis $\{e_i\}$.
Next, we consider the variation $\gamma(r,s)=\exp_{\gamma_0(s)}\left(r\eta(s)e_2(s) \right)$ for some smooth function $\eta$ to be determined. One easily computes that $\gamma_r=\eta(s)e_2(s)$, $\p_r L[\gamma]=0$, and 
\begin{equation*}
    \p_r^2 L[\gamma] = \int_{-s_0}^{s_0} \left((\eta')^2 - \eta^2 R(e_1, e_2,e_1,e_2) \right)ds .
\end{equation*}
So this variation gives that 
\begin{equation}\label{eq2.7}
    u_{22}(y_0,t_0) -u_{22}(x_0,t_0) - \vp' \int_{-s_0}^{s_0} \left((\eta')^2 - \eta^2 R(e_1, e_2,e_1,e_2) \right)ds \leq 0.
\end{equation}
Similarly, for $3\leq i\leq 2m$, the variation $\gamma(r,s)=\exp_{\gamma_0(s)}\left(r\zeta(s)e_i(s) \right) $ with $\zeta$ to be decided, produces
\begin{equation}\label{eq2.8}
     u_{ii}(y_0,t_0)-u_{ii}(x_0,t_0) - \vp' \int_{-s_0}^{s_0} \left((\zeta')^2 - \zeta^2 R(e_1, e_i,e_1,e_i) \right)ds \leq 0.
\end{equation}
Combining \eqref{eq2.6}, \eqref{eq2.7}, \eqref{eq2.8} together and in view of \eqref{eq2.3}, we obtain that
\begin{eqnarray} \label{eq2.9}
&& Q[u](y_0,t_0)-Q[u](x_0,t_0)  \\
 &=& \a(\vp')\left(u_{11}(y_0,t_0)-u_{11}(x_0,t_0) \right)+\b(\vp')\sum_{i=2}^{2m}\left(u_{ii}(y_0,t_0) -u_{ii}(x_0,t_0) \right) \nonumber \\
&\leq& 2\a(\vp')\vp'' +\b(\vp') \vp' \int_{-s_0}^{s_0}  \left((\eta')^2 - \eta^2 R(e_1, e_2,e_1,e_2) \right)ds \nonumber \\
&&+\b(\vp') \vp' \sum_{i=3}^{2m}\int_{-s_0}^{s_0} \left((\zeta')^2 - \zeta^2 R(e_1, e_i,e_1,e_i) \right)ds \nonumber 
\end{eqnarray}
Choosing $\eta(s)=\frac{c_{4\k}(s)}{c_{4\k}(s_0)}$ yields that 
\begin{eqnarray}\label{eq2.10}
    &&\int_{-s_0}^{s_0} \left((\eta')^2 - \eta^2 R(e_1, e_2,e_1,e_2) \right)ds \\
&=& \left.\eta \eta' \right|_{-s_0}^{s_0} -\int_{-s_0}^{s_0} \eta^2 \left(R(e_1, Je_1,e_1,Je_1) -4\k \right)ds \nonumber\\
&\leq & -2T_{4\k}(s_0),\nonumber
\end{eqnarray}
where we have used $\eta''+4\k \eta =0$ and the assumption $H\geq 4\k$. 
Similarly, we obtain with  $\zeta(s)=\frac{c_{\k}(s)}{c_{\k}(s_0)}$ that 
\begin{eqnarray}\label{eq2.11}
&& \sum_{i=3}^{2m} \int_{-s_0}^{s_0} \left((\zeta')^2 - \zeta^2 R(e_1, e_i,e_1,e_i) \right)ds \\ \nonumber
&=&  \int_{-s_0}^{s_0} \left(2(m-1)(\zeta')^2  -  \zeta^2 \Ric^\perp(e_1, e_1)\right) ds \\ \nonumber
&=& 2(m-1)\left.\zeta \zeta' \right|_{-s_0}^{s_0}  +\int_{-s_0}^{s_0} \zeta^2 \left(\Ric^\perp(e_1, e_1) -2(m-1)\k\right)ds \\
&\leq & -2(m-1)T_{\k}(s_0),\nonumber
\end{eqnarray}
where we have used $\zeta''+\k \zeta =0$ and $\Ric^\perp \geq 2(m-1)\k$. 
Inserting the above two estimates \eqref{eq2.10} and \eqref{eq2.11} into \eqref{eq2.9} gives 
\begin{equation}\label{eq2.12}
    Q[u](y_0,t_0)-Q[u](x_0,t_0) \leq 2 \L \vp.
\end{equation}
The desire inequality \eqref{eq2.0} follows immediately by combining \eqref{eq2.2} and \eqref{eq2.12}, thus completing the proof if $\p M =\emptyset$. 

Finally, let's deal with the situation that $M$ has a strictly convex boundary. 
Let $x_0$ and $y_0$ be such that the function 
$$Z(x,y)=u(y)-u(x)-2\vp\left(\frac{d(x,y)}{2}\right)$$
attains its maximum zero at $(x_0,y_0)$. 
We will rule out the possibility that either $x_0 \in \p M$ or $y_0 \in \p M$. Without loss of generality, we may assume $x_0 \in \p M$. Since $\p M$ is convex, there exists (see \cite{BGS02}) a length-minimizing geodesic $\gamma:[-s_0, s_0] \to M$ from $x_0$ to $y_0$ such that $\gamma(s)$ lies in the interior of $M$ for all $s \in (-s_0,s_0)$ and $g(\gamma'(-s_0), \nu(x_0))>0$, with $\nu(x_0)$ being the inward-pointing unit normal to $\p M$ at $x_0$. Then for $x(s)=\exp_{x_0}(s\nu(x_0))$, we have 
\begin{eqnarray*}
\frac{d}{ds} Z\left( x(s), y_0 \right) 
&=& -g(\n u, \nu(x_0) ) -\vp'(s_0) g(-\gamma'(-s_0), \nu(x_0)) \\ 
&=& \vp'(s_0) g(\gamma'(-s_0), \nu(x_0)) > 0. 
\end{eqnarray*}
This contradicts the fact that the function $Z(x,y)$ attains its maximum zero over $M\times M$ at $(x_0,y_0)$.
The proof is thus complete.  

\end{proof}

\begin{remark}
Theorem \ref{thm MC} remains valid if $u$ is only assumed to be a viscosity solution of \eqref{eq isotropic}. 
This will not be needed in the present paper and we refer the interested reader to \cite{Li16} and \cite{LW17} for how to prove modulus of continuity estimates for viscosity solutions. 
\end{remark}

\begin{remark}
One may follow the approach in \cite{Li20} to derive modulus of continuity estimates for fully nonlinear parabolic equations on K\"ahler manifolds. 
\end{remark}

\section{Lower Bounds for the First Nonzero Eigenvalue}

As in the Riemannian setting in \cite{AC13}, the large time behavior of the modulus of continuity estimates implies lower bound for the first nonzero eigenvalue of the Laplacian. The goal of this section is to prove lower bounds for the first nonzero eigenvalue of the Laplacian on a K\"ahler manifold stated in Theorem \ref{thm main}, as an application of the modulus of continuity derived in the previous section. 

Let's explain the approach of Andrews and Clutterbuck \cite{AC13} in more details. The idea to detect the first nonzero eigenvalue (with either $\p M =\emptyset$ or Neumann boundary condition) of the Laplacian is by knowing how quickly the solutions to the heat equation decay. 
This is because we may solve 
\begin{equation*}
    \begin{cases}
    u_t =\Delta u, &\\
    u(x,0) =u_0(x), &
    \end{cases}
\end{equation*}
by expanding $u_0 =\sum_{i=0}^\infty a_i \vp_i$, where $\vp_i$ are eigenfunctions of the Laplacian (with Neumann boundary condition if $\p M \neq \emptyset$). Then the solution to the heat equation is given by 
$$u(x,t) =\sum_{i=0}^\infty e^{-\mu_i t} a_i \vp_i.$$ 
This does not converges to zero as $\mu_0=0$, but the key observation is that $|u(x,t)-u(y,t)|$ does converges to zero, and in fact   
\begin{equation*}
    |u(x,t)-u(y,t)| \approx e^{-\mu_1 t} \text{ as } t \to \infty.
\end{equation*}
%$|u(x,t)-u(y,t)| \approx e^{-\mu_1 t}$ as $t \to \infty$. 
Thus the main step is to establish the estimate 
\begin{equation*}
    |u(x,t)-u(y,t)| \approx C e^{-\bar{\mu}_1 t}
\end{equation*}
for any solution to the heat equation. 
This turns out to be an easy consequence of the modulus of continuity estimates. Then taking $u(x,t)=e^{-\mu_1 t} \vp_1(x)$ leads to 
\begin{equation*}
    |\vp_1(x) -\vp_1(y)| \leq C e^{(\mu_1 -\bar{\mu}_1)t},
\end{equation*}
which implies $\mu_1 \geq \bar{\mu}_1$ by letting $t \to \infty$.

We begin with an immediate consequence of Theorem \ref{thm MC}, which asserts that if $\vp(\cdot, 0)$ lies above the modulus of continuity of $u(\cdot,0)$ and we evolve $\vp$ by \eqref{ODE K}, then $\vp(\cdot,t)$ lies above the modulus of continuity of $u(\cdot,t)$ for each positive $t$. %More precisely, we have 
\begin{corollary}\label{corollary 1}
Let $M$ and $u$ be the same as in Theorem \ref{thm MC}. Suppose $\vp: [0,D/2]\times [0,T) \to \R$ satisfies 
\begin{enumerate}
    \item $\vp_t \geq \L\vp$;
    %\a(\vp')\vp'' +\left(2(m-1)C_{\k}+C_{4\k} \right)\b(\vp')\vp'$; 
    \item $\vp' \geq 0$ on $[0,D/2]\times [0,T)$;
    \item $|u(y,0)-u(x,0)| \leq 2\vp \left(\frac{d(x,y)}{2},0 \right)$.
\end{enumerate}
Then 
\begin{equation*}
    |u(y,t)-u(x,t)| \leq 2\vp \left(\frac{d(x,y)}{2},t \right)
\end{equation*}
for all $x,y\in M$ and $t\in [0,T)$. 
\end{corollary}

\begin{proof}[Proof of Corollary \ref{corollary 1}]
For any $\e >0$, the function $\vp_\epsilon=\vp+\epsilon e^t$ satisfies 
$$(\vp_\e)_t > \L \vp_\e,$$ 
%$$(\vp_\e)_t > \a(\vp_\e')\vp_\e'' +\left(2(m-1)C_{\k}+C_{4\k} \right)\b(\vp_\e')\vp_\e',$$ 
so it cannot touch $\w$ from above by Theorem \ref{thm MC}. 
\end{proof}

On the interval $[0, D/2]$, we define the following corresponding one-dimensional eigenvalue problem with boundary conditions
$\phi(0)=0$ and $\phi'(D/2)=0$:
\begin{equation}
\bar{\sigma}_1(m, \k_1,\k_2, D/2)=\inf \left\{\frac{\int_0^{\frac D 2} |\phi'|^2 c_{\k_2}^{2m-2}c_{4\k_1} \, ds}{\int_0^{\frac D 2} |\phi|^2 c_{\k_2}^{2m-2}c_{4\k_1} \, ds}, \text{\quad with\quad} \phi(0)=0 \right\},
\end{equation}
where $c_\k$ is defined in \eqref{def c kappa}. 

\begin{lemma}\label{lmb2}
\begin{equation}
\bar{\mu}_1(m,\k_1,\k_2, D)=\bar{\sigma}_1(m, \k_1,\k_2, D/2).
\end{equation}
\end{lemma}
\begin{proof}
Let $\phi(s)$ be an eigenfunction on $[0, D/2]$ corresponding to $\bar{\sigma}_1(m,\k_1,\k_2, D/2)$. 
Then for $s\in[-D/2, 0)$ we define $\phi(s)$ by $\phi(s)=-\phi(-s)$. Clearly, $\phi(s)$ defined on $[-D/2, D/2]$ is a trial function
for $\bar{\mu}_1(m, \k_1,\k_2, D)$. Therefore
$$
\bar{\mu}_1(m, \k_1,\k_2, D)\le \bar{\sigma}_1(m,\k_1,\k_2, D/2).
$$

For the other direction, let $\psi(s)$ be an eigenfunction corresponding to $\bar{\mu}_1(m,\k_1,\k_2, D)$. Without of loss of generality we assume further that $\psi(s_0)=0$ for some $s_0\in [0, D/2]$. If $s_0>0$, we define $\psi(s)=0$ for $s\in[0, s_0)$. Then
$\psi(s)$ is a trial function for $\bar{\sigma}_1(m,\k_1,\k_2, D/2)$, and we have
\begin{eqnarray*}
\bar{\sigma}_1(m,\k_1,\k_2, D/2) &\leq & \frac{\int_0^{\frac D 2} |\psi'|^2 c_{\k_2}^{2m-2}c_{4\k_1} \, ds}{\int_0^{\frac D 2} |\psi|^2 c_{\k_2}^{2m-2}c_{4\k_1} \, ds} \\
&=& \frac{\int_{s_0}^{\frac D 2} |\psi'|^2c_{\k_2}^{2m-2}c_{4\k_1} \, ds}{\int_{s_0}^{\frac D 2} |\psi|^2 c_{\k_2}^{2m-2}c_{4\k_1} \, ds} \\
&=& \bar{\mu}_1(m,\k_1,\k_2, D).
\end{eqnarray*}
%$$
%\bar{\lambda}_1(m,\k_1,\k_2, \frac D 2)\le\frac{\int_0^{\frac D 2} |\psi'|^2 c_{\k_2}^{2m-2}c_{4\k_1} \, ds}{\int_0^{\frac D 2} |\psi|^2 c_{\k_2}^{2m-2}c_{4\k_1} \, ds}=\frac{\int_{s_0}^{\frac D 2} |\psi'|^2c_{\k_2}^{2m-2}c_{4\k_1} \, ds}{\int_{s_0}^{\frac D 2} |\psi|^2 c_{\k_2}^{2m-2}c_{4\k_1} \, ds}=\bar{\mu}_1(m,\k_1,\k_2, D),
%$$
This proves the lemma.
\end{proof}

\begin{lemma}\label{lmvp}
There exists an odd eigenfunction  $\phi$ corresponding to $\bar{\mu}_1(m,\k_1,\k_2, D)$ satisfying
$$
\phi'' - (T_{4\k_1}+2(m-1)T_{\k_2}) \phi' +\bar{\mu}_1(m, \k_1,\k_2, D)\phi =0
$$
in $(0, D/2)$ with  $\phi'(s)>0$ in $(0, D/2)$ and $\phi'(D/2)=0$.

\end{lemma}
\begin{proof}
Lemma \ref{lmb2} gives the existence of odd eigenfunction $\phi$ corresponding to $\bar{\mu}_1(m,\k_1,\k_2, D)$ with boundary conditions $\phi(0)=0$ and $\phi'(D/2)=0$,  which does not change sign in $(0,D/2]$. 
Since
$$
\Big(c_{\k_2}^{2m-2}c_{4\k_1} \phi'\Big)'=-\bar{\mu}_1(m,\k_1,\k_2, D) c_{\k_2}^{2m-2}c_{4\k_1} \phi,
$$
we can choose  $\phi(s)$ such that $\phi'(s)>0$ in $(0, D/2)$. 
\end{proof}

We provide the proof of Theorem \ref{thm main} below for the sake of completeness. 

\begin{proof}[Proof of Theorem \ref{thm main}]

For any $D_1>D$, let $\bar{\mu}_1=\bar{\mu}_1(m,\k_1,\k_2,D_1)$ 
be the first nonzero Neumann eigenvalue 
of the operator $\L$ defined in \eqref{eq 1D} on the interval $[-D_1/2,D_1/2]$. By Lemma \ref{lmvp}, the associated eigenfunction $\phi(s)$ is odd and can be chosen so that $\phi>0$ on $(0,D_1/2]$ and $\phi'(0)>0$. 
Let $u(x)$ be an eigenfunction corresponding to $\mu_1$.
Then there exists $C>0$ such that $$
u(y)-u(x)-2 C \phi\left(\frac{d(x,y)}{2}\right)\le 0
$$ for all $x,y\in M$. 
Direct calculation shows that the functions 
$v(x,t):=e^{- \mu_1 t} u(x)
$ and $
\vp(s,t):= C \, e^{-\bar{\mu}_1 t}\phi
$
satisfy
$$
\frac{\p v}{\p t}=\Delta v,
$$
and
$$ \frac{\p \vp}{\p t}=\L \vp,$$
respectively. 
%(see also Remark \ref{remark 1} and \ref{remark regularity}),
Moreover, it's easy to verify that $C \vp$ satisfies all the conditions in Corollary \ref{corollary 1}, and therefore, we have 
\begin{equation*}
    u(y,t)-u(x,t) \leq 2 C \vp\left(\frac{d(x,y)}{2}, t \right), 
\end{equation*}
or equivalently,
\begin{equation}\label{eq3.9}
 e^{- \mu_{1} t} (u(y)-u(x))\le 2 C\,  e^{-\bar{\mu}_1 t}\phi\left(\frac{d(x,y)}{2}\right)
\end{equation}
for all $x,y\in M$ and $t>0$. Thus as $t\rightarrow \infty$, inequality \eqref{eq3.9} implies
$$
\mu_{1}\ge \bar{\mu}_1.
$$
Theorem \ref{thm main} then follows easily because $\bar\mu_1 \to \bar{\mu}_{1}(m,\k_1,\k_2,D)$ as $D_1 \to  D$.
\end{proof}

\begin{remark}
The modulus of continuity approach can be adapted to prove similar lower bounds for the first nonzero eigenvalue the $p$-Laplacian for Riemannian manifolds when $1<p\leq 2$. This is outlined in \cite[Section 8]{Andrewssurvey15} (see also \cite[Section 2]{LW19eigenvalue}). It's not hard to adapt our argument to prove 
lower bounds for the first nonzero eigenvalue of the $p$-Laplacian on K\"ahler manifolds when $1<p\leq 2$. The $p>2$ case can be handled with gradient estimates method (see \cite{NV14}\cite{Valtorta12}\cite{LW19eigenvalue2}) for Riemannian manifolds, but remains open for K\"ahler manifolds. 
\end{remark}

\section{Some Explicit Lower Bounds}

In this section, we prove the explicit lower bounds and estimates claimed in Proposition \ref{prop explicit bounds}. Then we compare our lower bound in Theorem \ref{thm main} with Lichnerowicz's lower bound. 

\begin{proof}[Proof of Proposition \ref{prop explicit bounds}]
For the $\k_1=\k_2=0$ case, simply observe that the corresponding first eigenfunction is given by $\sin\left(\frac{\pi t}{D}\right)$. 

If $\k_1>0$ and $\k_2=0$, then by a result of Tsukamoto \cite{Tsukamoto57}, the diameter is bounded from above by $\frac{\pi}{2\sqrt{\k_1}}$, which yield the inequality by domain monotonicity of the first Neumann eigenvalue. The inequality follows by observing that the first eigenfunction on the interval $[-\frac{\pi}{4\sqrt{\k_1}}, \frac{\pi}{4\sqrt{\k_1}}]$ is given by $\sin\left(2\sqrt{\k_1}t \right)$. 

Similarly, if $\k_1=0$ and $\k_2 >0$, then $\Ric^\perp \geq 2(m-1)\k_2$ implies $D\leq \frac{\pi}{\sqrt{\k_2}}$ by \cite[Theorem 3.2]{NZ18} and the first eigenfunction on $[-\frac{\pi}{2\sqrt{\k_2}}, \frac{\pi}{2\sqrt{\k_2}}]$ is $\sin\left(\sqrt{k_2} t\right)$ with first nonzero eigenvalue equal to $(2m-1)\k_2$. 
%Finally, the inequality $\bar{\mu}_1(m,\k,\k,D) >  \bar{\sigma}_1(2m,\k,D)$ is a consequence of ODE comparison noticing that $T_{4\k} \leq  4 T_\k$ if $\k < 0$. 
\end{proof}

It's interesting to compare Theorem \ref{thm main} with Lichnerowicz's lower bound in the positive curvature case. For exmaple, if $H\geq 4\k_1>0$ and $\Ric^\perp \geq 0$, then Proposition \ref{prop explicit bounds} gives 
\begin{equation*}
    \mu_1(m,\k_1,0,D) \geq \bar{\mu}_1\left(m,\k_1,0, \frac{\pi}{2\sqrt{\k_1}}\right)=8\k_1,
\end{equation*}
while Lichnerowicz's lower bound for K\"ahler manifolds yields $\mu_1 \geq 8\k_1$ since $\Ric \geq 4\k_1 >0$. Therefore, we have better (strictly better for any $D$ smaller than the maximal diameter $\frac{\pi}{2\sqrt{\k_1}}$ allowed under the condition $H\geq 4\k_1 >0$ by \cite{Tsukamoto57}) lower bounds under stronger curvature assumptions.  

\section{A Comparison Theorem for $d(x,\p M)$}

Let $(M^m,g,J)$ be a compact K\"ahler manifold with smooth nonempty boundary $\p M$. Denote by $d(x,\p M)$ the distance function to the boundary of $M$ given by 
\begin{equation*}
    d(x,\p M)=\inf\{d(x,y) : y \in \p M \}.  
\end{equation*}
Recall the inradius $R$ is given by 
\begin{equation*}
    R=\sup \{d(x,\p M) :x \in \p M \}.
\end{equation*}
In this section, we prove the following comparison theorem for the second derivatives of $d(x,\p M)$ on a K\"ahler manifold with boundary. 
%One should compare it with its Riemannian version \cite{}. 

%We will be able to deal with arbitrary lower bound of the mean curvature of the boundary.

%A related comparison theorem for $d(x,y)$ was proved in \cite[Theorem 3]{AC13} and played the key role in deriving the modulus of continuity estimate when either $\p M$ is empty or $\p M$ is convex (with the Neumann boundary condition).  
%For $N \in [n,\infty)$, the comparison theorem states 
\begin{thm}\label{Thm comparison distance to boundary}
Let $(M^m,g,J)$ be a compact K\"ahler manifold with smooth nonempty boundary $\p M$. 
Suppose that $H\geq 4\k_1$ and $\Ric^\perp \geq 2(m-1)\kappa_2$ for some $\k_1,\k_2 \in \R$, and the second fundamental form on $\p M$ is bounded from below by $\Lambda \in \R$. 
Let $\vp:[0, R ] \to \R_+$ be a smooth function with $\vp' \geq 0$. Then the function $v(x) =\vp\left(d(x,\p M)\right)$ is a viscosity supersolution of 
\begin{equation*}
    Q[v] =\left. \left[\a (\vp')\vp'' -\b(\vp')\vp'\left(2(m-1)T_{\k_2, \Lambda}+T_{4\k_1, \Lambda} \right) \right] \right|_{d(x,\p M)},
\end{equation*}
on $M$, where $Q$ is the operator defined in \eqref{eq isotropic}.
\end{thm}
\begin{proof}[Proof of Theorem \ref{Thm comparison distance to boundary}]
By approximation, it suffices to consider the case $\vp' >0$ on $[0,  R]$. 
By definition of viscosity solutions (see \cite{CIL92}), it suffices to prove that for any smooth function $\psi$ touching $v$ from below at $x_0 \in M$, i.e., 
\begin{align*}
    \psi(x) \leq v(x) \text{ on } M, \text{\quad}
    \psi(x_0) &= v(x_0),
\end{align*}
it holds that
\begin{equation*}
    Q[\psi](x_0) \leq \left. \left[\a (\vp')\vp'' -\b(\vp')\vp'\left(2(m-1)T_{\k_2, \Lambda}+T_{4\k_1, \Lambda} \right) \right] \right|_{d(x_0,\p M)}.
\end{equation*}
Since the function $d(x, \p M)$ may not be smooth at $x_0$, so we need to replace it by a smooth function $\bar{d}(x)$ defined in a neighborhood $U(x_0)$ of $x_0$ satisfying $\bar{d}(x) \geq d(x, \p M)$ for $x \in U(x_0)$ and $\bar{d}(x_0)=d(x_0, \p M)$. 
The construction is standard (see e.g. \cite[pp. 73-74]{Wu79} or \cite[pp. 1187]{AX19}), which we state below for reader's convenience. 

Since $M$ is compact, there exists $y_0 \in \p M$ such that $$d(x_0,y_0)=d(x_0, \p M):=s_0.$$ Let $\gamma:[0,s_0] \to M$ be the unit speed length-minimizing geodesic with $\gamma(0)=x_0$ and $\gamma(s_0)=y_0$.  
For any vector $X \in \exp^{-1}_{x_0}U(x_0)$, let $X(s), s\in [0,s_0]$ be the vector field obtained by parallel translating $X$ along $\gamma$, and decompose it as 
\begin{equation*}
    X(s) =a X^\perp (s) +b \gamma'(s)+cJ\g'(s),
\end{equation*}
where $a, b$ and $c$ are constants along $\gamma$ with $a^2+b^2+c^2 =|X|^2$, and $X^\perp(s)$ is a unit parallel vector field along $\gamma$ orthogonal to $\gamma'(s)$ and $J\gamma'(s)$. Define 
\begin{equation*}
    W(s)=a \, \eta(s) X^\perp(s) + b\left(1-\frac{s}{s_0} \right)\gamma'(s)+c\, \zeta(s)J\gamma'(s),
\end{equation*}
where $\eta, \zeta:[0,s_0] \to \R_+$ are two $C^2$ functions to be chosen later. 
Next we define the $n$-parameter family of curves $\gamma_X :[0,s_0] \to M$ such that 
\begin{enumerate}
    \item $\gamma_0 =\gamma$;
    \item $\gamma_X(0) =\exp_{x_0}(W(0))$ and $\gamma_X(s_0) \in \p M$;
    \item $W(s)$ is induced by the one-parameter family of curves $l \to \gamma_{lX}(s)$ for $l \in [-l_0,l_0]$ and $s\in [0,s_0]$;
    \item $\gamma_X$ depends smoothly on $X$.
\end{enumerate}
Finally let $\bar{d}(x)$ be the length of the curve $\gamma_X(x)$ where $x=\exp_{x_0}(X) \in U(x_0)$. 
Then we have
$\bar{d}(x) \geq d(x,\p M)$ on $U(x_0)$, $\bar{d}(x_0)=d(x_0, \p M)$. 

Recall the  first and second variation formulas:
\begin{equation*}
\n \bar{d}(x_0) =-\gamma'(0) 
\end{equation*}
and
\begin{eqnarray*}
    \n^2 \bar{d} (X,X)&=& -a^2 \eta(s_0)^2 A(X^\perp(s_0), X^\perp(s_0))+c^2 \zeta(s_0)^2 A(J\gamma',J\gamma') \\
    && +a^2 \int_0^{s_0} \left((\eta')^2 -\eta^2 R(X^\perp, \gamma',X^\perp, \gamma') \right) ds\\
    && +c^2 \int_0^{s_0} \left((\zeta')^2 -\zeta^2 R(J\gamma', \gamma',J\gamma', \gamma') \right) ds
\end{eqnarray*}
where $A$ denotes the second fundamental form of $\p M$ at $y_0$. Then for an orthonormal frame $\{e_i(s)\}_{i=1}^{2m}$ along $\gamma$ with $e_1(s) =\gamma'(s)$ and $e_2(s)=J\gamma'(s)$, we have
\begin{equation}\label{1st bard}
\n \bar{d}(x_0) =-e_1(0),
\end{equation}
and
\begin{equation}\label{2nd bardn}
\n^2 \bar{d} (e_1(0),e_1(0))=0.
\end{equation}
For $i=2$, we obtain by choosing $\zeta(s)=C_{4\kappa_1, \Lambda}(s_0-s) /C_{4\kappa_1, \Lambda}(s_0)$ with $C_{4\kappa_1, \Lambda}$ defined in \eqref{C def} that 
\begin{eqnarray}%\label{2nd bardi}
    && \n^2 \bar{d} (e_2(0),e_2(0)) \\ \nonumber 
    &=& - \zeta(s_0)^2 A(e_2(s_0), e_2(s_0)) + \int_0^{s_0} (\zeta')^2 -\zeta^2 R(J\gamma', \gamma',J\gamma',\gamma') \, ds \\ \nonumber
    &\leq& -\frac{A(J\gamma'(s_0), J\gamma'(s_0))}{C_{4\k_1, \Lambda}(s_0)^2} +\int_0^{s_0} (\zeta')^2 -4\k_1 \zeta^2 \, ds \\\nonumber
    &=& -\frac{A(J\gamma'(s_0), J\gamma'(s_0))}{C_{4\k_1, \Lambda}(s_0)^2} +\frac{\Lambda}{C_{4\k_1, \Lambda}(s_0)^2}-T_{4\k_1, \Lambda}(s_0)\\
    &\leq& -T_{4\k_1, \Lambda}(s_0). \nonumber
\end{eqnarray}
For $3\le i \le 2m$, we have 
\begin{eqnarray*}
    \n^2 \bar{d} (e_i(0),e_i(0)) 
    =- \eta(s_0)^2 A(e_i(s_0), e_i(s_0)) + \int_0^{s_0} (\eta')^2 -\eta^2 R(e_i, \gamma',e_i, \gamma') \, ds.
\end{eqnarray*}
Summing over $3\leq i \leq 2m$ and choosing $\eta(s)=C_{\k_2, \Lambda}(s_0-s) /C_{\k_2, \Lambda}(s_0)$ with $C_{\k_2, \Lambda}$ defined in \eqref{C def} gives 
\begin{eqnarray}%\label{2nd bardi} 
     && \sum_{i=3}^{2m} \n^2 \bar{d} (e_i(0),e_i(0)) \\ \nonumber 
     &=& - \frac{\sum_{i=3}^{2m}A(e_i(s_0), e_i(s_0))}{C_{\k_2, \Lambda}(s_0)^2} + \int_0^{s_0} 2(m-1)(\eta')^2 -\eta^2 \Ric^\perp(\gamma',\gamma') \, ds \\\nonumber
    &\leq& - \frac{\sum_{i=3}^{2m}A(e_i(s_0), e_i(s_0))}{C_{\k_2, \Lambda}(s_0)^2} +\frac{2(m-1)\Lambda}{C_{\k_2, \Lambda}(s_0)^2}-2(m-1)T_{\k_2, \Lambda}(s_0) \\
    &\leq& -2(m-1)T_{\k_2, \Lambda}(s_0). \nonumber
\end{eqnarray}
Since the function $\psi(x) -\vp\left(d(x, \p M)\right)$ attains its maximum at $x_0$ and $\vp' >0$, it follows that the function $\psi(x) -\vp(\bar{d}(x))$ attains a local maximum at $x_0$. The first and second derivative tests yield
$$ \n  \psi (x_0) =-\vp' e_1 (0), \quad \psi_{11}(x_0)  \leq  \vp'',$$
    and
    $$
    \psi_{ii}(x_0)  \leq \vp' \n^2\bar{d} \left(e_i(0), e_i(0)\right)
$$
for $2 \le i \leq 2m$, where we used (\ref{1st bard}) and (\ref{2nd bardn}).
Here and below the derivatives of $\vp$ are all evaluated at $s_0=d(x_0, \p M)$. 
%It then follows from (\ref{2nd bardi}) that
%\begin{eqnarray*}
%\sum_{i=2}^{2m} \n^2 \bar{d}\left(e_i(0), e_i(0)\right) &=&- \eta(s_0)^2 H(y_0) + \int_0^{s_0} (\eta')^2 -\eta^2 R(J\gamma', \gamma',J\gamma', \gamma) ds \\
%&& + \int_0^{s_0}2(m-1)(\eta')^2 -\eta^2 \Ric^\perp(\gamma', \gamma')  \, ds,
%\end{eqnarray*}
Thus we have 
\begin{eqnarray}\label{eq 3.1} 
     Q[\psi](x_0) &=& \a (\vp')\psi_{11} +\b(\vp') \sum_{i=2}^{2m} \psi_{ii} \\\nonumber
    &\leq& \a (\vp')\vp'' +\b(\vp')\vp'\left(\sum_{i=2}^{2m}  \n^2 \bar{d}(e_i(0),e_i(0)) \right) \\
    &\leq & \a (\vp')\vp'' -\b(\vp')\vp'\left(2(m-1)T_{\k_2, \Lambda}+T_{4\k_1, \Lambda} \right).\nonumber
\end{eqnarray}
The proof is complete. 
\end{proof}
\section{Lower Bounds for the First Dirichlet Eigenvalue}
Let $\bar{\l}_1:= \bar{\lambda}_1(m, \k_1,\k_2, \Lambda, R)$ denote the first eigenvalue of one-dimensional eigenvalue problem \begin{equation*}\label{}
    \begin{cases}
   \vp''-\left(2(m-1)T_{\k_2, \Lambda}+T_{4\k_1, \Lambda} \right)\vp'  =-\l\vp, \\
    \vp(0)=0,     \vp'(R)=0.
    \end{cases}
\end{equation*}
where $T_{\k,\Lambda}$ is defined in \eqref{def T kappa Lambda}. 
It's easy to see that 
\begin{equation*}
\bar{\lambda}_1=\inf \left\{\frac{\int_0^{R} |\phi'|^2 C_{\k_2,\Lambda}^{2m-2}C_{4\k_1,\Lambda} \, ds}{\int_0^{R} |\phi|^2 C_{\k_2,\Lambda}^{2m-2}C_{4\k_1,\Lambda} \, ds}, \text{\quad with\quad} \phi(0)=0 \right\},
\end{equation*}
where $C_{\k,\Lambda}$ is defined in \eqref{C def}. 

\begin{proof}[Proof of Theorem \ref{Thm C}]
Let $\vp$ be an eigenfunction with respect to
$\bar{\l}_1$, then 
$$
\vp''-\left(2(m-1)T_{\k_2, \Lambda}+T_{4\k_1, \Lambda} \right)\vp'  =-\bar{\l}_1\vp, 
$$
with $ \vp(0)=0$ and $\vp'(R)=0$. Moreover $\vp$ can be chosen so that $\vp>0$ on $(0, R]$ and $\vp'>0$ on $[0,R)$, see Lemma \ref{lmvp}. Define a testing function
$$
v(x):=\vp(d(x,\p M)).
$$
Then using Theorem \ref{Thm comparison distance to boundary} (choosing $\a =\b =1$), we have
$$
\Delta v(x) \le -\bar{\l}_1 v(x),
$$
away from the cut locus of $\p M$, and thus globally in the distributional sense. 
Notice that $v(x)>0$ in $M$ and $v(x)=0$ on $\p M$, then
we conclude from  Barta’s inequality (see \cite[Lemma 1.1]{Kasue84}) that
$$
\l_1 \ge \bar{\l}_1.
$$
The proof of Theorem \ref{Thm C} is complete.
\end{proof}

%\begin{remark}
%With a generalized Barta's inequality for the $p$-Laplacian (see \cite[Theorem 3.1]{LW20a}), one can easily derive similar lower bounds for the first Dirichlet eigenvalue of the $p$-Laplacian for all $1<p<\infty$. 
%\end{remark}

%\input{0Sharpness}
%\input{0Introduction}
%\input{0Modulus}
%\input{0Eigenvalue}

%\input{0Sharpness}

%\input{0Quaternion}

%\input{0p-Eigenvalue}

%\input{0Calculations}
%\input{1intro}

%\input{2modulus}
%\input{2 comparison new}
%\input{3MCEstimates}
%\input{4Eigenvalue}
%\input{5Gradient}
%\input{6models}
%\input{7Proof}
%\input{8sharpness}
%\input{Dirichlet and Robin}
%\input{5comparison to bdy}
%\input{ComparionThm}
%\input{Quaternion}

%\input{2comparison}
%\input{SecondRobin_preprint.tex}

%\tableofcontents

%\newpage
%\bigskip
%\newpage
\section*{Acknowledgments} {We would like to thank Professors Ben Andrews, Lei Ni and Richard Schoen for helpful discussions. 
We are also grateful to Professor Guofang Wei for pointing out the reference \cite{BS19}}. 

\bibliographystyle{alpha}
\bibliography{ref}

\end{document}